\documentclass[11pt, reqno]{amsart}

\usepackage{amssymb,latexsym,amsmath,amsfonts,amsthm}
\usepackage{mathrsfs}
\usepackage{enumitem}
\usepackage[usenames]{color}
\usepackage{hyperref}

\allowdisplaybreaks

\numberwithin{equation}{section}

\theoremstyle{definition}
\newtheorem{definition}{Definition}[section]

\theoremstyle{remark}

\theoremstyle{plain}
\newtheorem{theorem}[definition]{Theorem}
\newtheorem{result}[definition]{Result}
\newtheorem{lemma}[definition]{Lemma}

\newtheorem{example}[definition]{Example}

\newcommand{\eps}{\varepsilon}
\newcommand{\zt}{\zeta}

\newcommand\pd[3]{\frac{\partial^{{#3}}{#1}}{\partial{#2}}}
\newcommand{\laplc}{\triangle}

\newcommand{\bdy}{\partial}
\newcommand{\OM}{\Omega}
\newcommand{\D}{\mathbb{D}}
\newcommand{\ome}{\omega}

\newcommand\ball[1]{\mathbb{B}^{#1}}
\newcommand{\apprh}{\mathscr{A}_{\alpha,\,N}}
\newcommand{\loc}{\OM_F\cap\Delta}

\newcommand{\smoo}{\mathcal{C}}
\newcommand{\hol}{\mathcal{O}}
\newcommand\leb[1]{\mathbb{L}^{{#1}}}

\newcommand{\bcdot}{\boldsymbol{\cdot}}

\newcommand{\lrarw}{\longrightarrow}

\newcommand{\lamf}{\Lambda_f}
\newcommand\met[1]{ds^2_{{#1}}}
\newcommand{\jay}{{J}}

\newcommand\lint[2]{\int\limits_{{#1}}^{{#2}}}
\newcommand\geef[1]{G_f^{\raisebox{1pt}{$\scriptstyle\,{#1}$}}}
\newcommand{\func}{\phi_{t,\,\xi}}
\newcommand{\funk}{\varphi_{z,\,t,\,\xi}}

\newcommand{\zahl}{\mathbb{Z}}
\newcommand{\nat}{\mathbb{N}}

\newcommand{\CC}{\mathbb{C}^2}
\newcommand{\C}{\mathbb{C}} 
\newcommand{\R}{\mathbb{R}}

\newcommand{\re}{{\sf Re}}
\newcommand{\im}{{\sf Im}}

\begin{document}

\title[Growth of the Bergman metric]{On the growth of the Bergman metric \\
near a point of infinite type}

\author{Gautam Bharali}
\address{Department of Mathematics, Indian Institute of Science, Bangalore 560012, India}
\email{bharali@iisc.ac.in}

\begin{abstract}
We derive optimal estimates for the Bergman kernel and the Bergman metric for certain model
domains in $\CC$ near boundary points that are of infinite type. Being unbounded models, these
domains obey certain geometric constraints\,---\,some of them necessary for a non-trivial Bergman
space. However, these are \emph{mild} constraints: unlike most earlier works on this subject, we
are able to make estimates for \emph{non-convex} pseudoconvex models as well. In fact,
the domains we can analyse range from being mildly infinite-type to very flat at
infinite-type boundary points.
\end{abstract}


\maketitle

\vspace{-0.2cm}
\section{Introduction}\label{S:intro}
Let $\OM\subset\CC$ be a pseudoconvex domain (not necessarily bounded) having a 
$\smoo^\infty$-smooth boundary.
Let $p\in\bdy\OM$ be a point of infinite type: i.e., for each $N\in\zahl_+$,
there exists a germ of a $1$-dimensional complex-analytic variety through $p$ whose order of 
contact with $\bdy\OM$ at $p$ is at least $N$. If $\bdy\OM$ is not Levi-flat around $p$, then
there exist local holomorphic coordinates $(z,w; U_p)$ centered at $p$ such that
\begin{equation}\label{E:local}
  \OM\cap U_p\,=\,\{(z,w)\in U_p:\im w > F(z)+R(z,\re w)\},
\end{equation}
where $F$ is a smooth, subharmonic, non-harmonic function defined in a neighbourhood of $z=0$
that vanishes to infinite order at $z=0$; $R(\bcdot \ ,0)\equiv 0$; and 
$R$ is $O(|z||\re w|,|\re w|^2)$. 
Given the infinite order of vanishing of $F$ at $z=0$, many of the
ideas for estimating the growth the Bergman kernel and its partial
derivatives\,---\,evaluated on the diagonal\,---\,as one approaches a finite-type boundary point are
no longer helpful. But some of the ideas alluded to can be useful (see, e.g., item~(2) below)
if the function $F$ introduced in \eqref{E:local} is the restriction of a global subharmonic, non-harmonic
function. Such a function gives us a model domain
\begin{equation}\label{E:mod}   
  \OM_F\,:=\,\{(z,w)\in\CC:\im w > F(z)\},
\end{equation}
which approximates $\bdy\OM$ to infinite order along the complex-tangential directions at $p$. This paper
studies the growth the Bergman kernel (evaluated on the diagonal) and the
Bergman metric on $\OM_F$ as one approaches $(0,0)\in \CC$, with certain reasonable conditions on
$F$ so that:
\begin{itemize}[leftmargin=22pt]
  \item the Bergman space for $\OM_F$\,---\,which we denote by $A^2(\OM_F)
  :=\leb{2}(\OM_F, \C)\cap\hol(\OM_F)$\,---\,is non-trivial; and
  \item the problem just described is tractable despite the difficulties arising from $F$ vanishing to infinite
  order at $z=0$.
\end{itemize}
\vspace{0.3mm}

The model domains defined by \eqref{E:mod} are reminiscent of the domains studied in
\cite{bharali:gBknitp11} but, in fact, we shall study a much wider class of model domains
than those introduced in \cite{bharali:gBknitp11}. To elaborate: the domains studied
in the latter paper satisfied a condition $(*)$\,---\,refer to
\cite[page~2]{bharali:gBknitp11}\,---\,which involved a technical growth condition that turns out to
be unnecessary. For the domains $\OM_F$ that we consider, in this paper $F$ will just be a radial
function. I.e., it will satisfy the condition
\begin{itemize}
  \item[$(\bullet)$] $F(z)=F(|z|)$ $\forall z\in\C$.
\end{itemize}
While the condition~$(\bullet)$ limits the sorts of domains of the form \eqref{E:mod}
that we wish to study, there are two reasons for restricting our attention to the case where
$F$ is radial:
\begin{enumerate}[leftmargin=25pt]
  \item[(1)] A recurring technique for obtaining the kind of estimates that we seek
  is the use of scaling: information on, say, the Bergman kernel at the unit scale is classical,
  while an understanding of $K_{\OM}(z,w)$ as
  $\OM\!\ni\!(z,w)\to (0,0)$ (where $(0,0)\in \bdy\OM$) is obtained by rescaling appropriately
  to unit scale: see,
  for instance, \cite{diedHerOhs:Bkuepd86} by Diederich \emph{et al.},
  \cite{nagelRosaySteinWainger:eBkSkcwpd88} and \cite{nagelRosaySteinWainger:eBSkC289}
  by Nagel \emph{et al.}, \cite{mcneal:bbBkfC289} by McNeal. These methods
  do not seem to yield optimal estimates, even just for model domains of the form 
  \eqref{E:mod}, if $F$ vanishes to infinite order at $z=0$ and 
  $F$ behaves differently along different \emph{real} directions in $\C$.
  The work of Kim--Lee \cite{kimLee:abBkaicitpd02}\,---\,who examine a class
  of \emph{convex} domains that form a proper subclass of the class of domains
  we shall study\,---\,suggests strongly that our problem is more tractable if
  $F$ is radial.
  \vspace{0.5mm}
  
  \item[(2)] Once we assume that $F$ is radial and $\bdy\OM_{F}$ is not Levi-flat, it
  follows that $F(z) > 0$ $\forall z\in\C\setminus\{0\}$: see part~$(a)$ of
  Theorem~\ref{T:kernel} below. Then (provided one has a localisation theorem for
  the Bergman kernel for $\OM_F$) the arguments of Boas \emph{et al.} in
  \cite{boasStraubeYu:blBkm95} imply that information on the growth of
  the Bergman kernel or the Bergman metric
  for $\OM_F$ yields analogous information for $\OM$ as one approaches $p$
  through
  $\OM\cap U_p$, where $U_p$ is as introduced
  by \eqref{E:local} and the pair $(\OM, p)$ satisfies the assumptions
  stated prior to \eqref{E:mod}.
\end{enumerate}
The function $K_{\OM}$ introduced above is defined as follows: if, for a domain
$\OM\subset \CC$, $B_{\OM} : \OM\times \OM\lrarw \C$ denotes the Bergman kernel
for $\OM$, then $K_{\OM}(z,w) := B_{\OM}\big((z,w), (z,w)\big)$. We will abbreviate
$K_{\OM_F}$ as $K_F$.
\smallskip

What enables us to so significantly weaken the condition $(*)$ in
\cite[page~2]{bharali:gBknitp11} to $(\bullet)$ above, and yet expect non-trivial results, is a
localisation principle for the Bergman kernel and the Bergman metric
by Chen \emph{et al.} \cite{chenKamimotoOhsawa:bBki03}: see Section~\ref{S:BergPrelim}
for details. 
\smallskip

With these ingredients, we get the \emph{optimal} expressions for the growth of the quantities
considered\,---\,as the inequalities \eqref{E:approach} and \eqref{E:approach_met} below show.
We briefly summarise where those inequalities hold:
\begin{enumerate}[leftmargin=25pt]
  \item[$(i)$] We get upper bounds on $K_F$ and on the Bergman metric for $\OM_F$
  that hold in a family of approach regions for $(0,0)\in \bdy\OM_F$ comprising regions
  with \emph{arbitrarily high} orders of contact with $\bdy\OM_F$ at $(0,0)$, our bounds being
  independent of the approach region.
  \vspace{0.5mm}
  
  \item[$(ii)$] There exists an $\overline{\OM}_F$-open neighbourhood $\ome$ of
  $(0,0)$ such that our lower bound for $K_F$ holds true
  on $\ome\cap\OM_F$. 
\end{enumerate}
It is well known that, \emph{even if} $F$ is radial, $K_F(z,w)\gtrsim\|(z,w)\|^{-2}$ is the best
that one expects (for non-tangential approach) without any additional information on $F$. For
instance: with the additional information that $(0,0)\in\bdy\OM_F$ is of finite type, we get
optimal estimates because, in this case, we can find constants $C, r>0$, and $M\in\zahl_+$
such that
\begin{align*}
  \ball{2}(0;r)\cap\{(z,w):\im{w}>C|z|^{2M}\}
  &\subset \OM_F\cap \ball{2}(0;r) \notag \\
  &\subset \ball{2}(0;r)\cap\{(z,w):\im{w}>(1/C)|z|^{2M}\}.
\end{align*}
Here, we can make precise estimates by exploiting the simplicity of the prototypal
function $z\longmapsto |z|^{2M}$. When $F$ vanishes to infinite
order at $0$, there is no obvious notion of a prototype for $F$. However, $F$ exhibits,
in some sense, a ``controlled infinite-order vanishing'' at $0$ if it satisfies the condition
stated right after part~$(a)$ of Theorem~\ref{T:kernel}. This condition is motivated 
by the fact that it encompasses a very large class of
domains, ranging from the ``mildly infinite-type'' to the very flat at $(0,0)$: see
the examples in Section~\ref{S:examples}. To state this condition, we need the following:

\begin{definition}\label{D:doubling}
An increasing function $g: [0,R]\lrarw \R$ is said to satisfy a \emph{doubling condition} if
$g(0) = 0$ and there exists a constant $\sigma > 1$ such that
\[
  2g(x)\,\leq\,g(\sigma x) \; \; \forall x\in [0, R/\sigma].
\]
We will call the constant $\sigma>1$ a \emph{doubling constant} for $g$.
\end{definition}

We shall also need the following notation. Let 
$f:[0,\infty)\lrarw\R$ be a strictly increasing function and let $f(0)=0$. We
define the function $\Lambda_f$ as
\[
\Lambda_f(x) \ := \ \begin{cases}
			-1/\log(f(x)), &\text{if $0<x<f^{-1}(1)$}, \\
			0, &\text{if $x=0$}.
			\end{cases}
\]

\begin{theorem}\label{T:kernel}
Let $F$ be a $\smoo^\infty$-smooth subharmonic function on $\C$ that vanishes
to infinite order at $0\in\C$ and is radial (i.e., satisfies condition~$(\bullet)$ above). Suppose
the boundary of the domain $\OM_F:=\{(z,w)\in\CC:\im w > F(z)\}$ is not Levi-flat around 
$(0,0)$.
\begin{enumerate}[leftmargin=25pt]
  \item[$(a)$] Let $f$ be given by the relation $f(|z|)=F(z)$
  $\forall z\in \C$. Then, $f$ is a strictly
  increasing function on $[0,\infty)$ and $\lim_{r\to \infty}f(r) = +\infty$.
\end{enumerate}

Assume that there exists  a constant $R\in \big(0, f^{-1}(1)\big)$
such that $\lamf|_{[0, R]}$ satisfies a doubling condition. Then:
\begin{enumerate}[leftmargin=25pt]
  \item[$(b)$] There exists a constant $C_1 > 0$ and, for each $\alpha>0$ and
  $N\in\zahl_+$, there exists a constant $r(\alpha, N)>0$ such that
  \begin{align}\label{E:approach}
    (1/C_1)(\im w)^{-2}\big(f^{-1}(\im w)\big)^{-2}\leq\,K_F&(z,w)\,\leq\,C_1(\im w)^{-2}
    \big(f^{-1}(\im w)\big)^{-2} \\
    &\forall (z,w)\in\apprh\cap\{(z,w): \im{w}<r(\alpha, N)\}, \notag
  \end{align}
  where $\apprh$ denotes the approach region
  \[
    \apprh\,:=\,\big\{(z,w)\in\OM_F:\sqrt{|z|^2+|\re w|^2}<\alpha(\im w)^{1/N}\big\}.
  \]
  \item[$(c)$] Furthermore, there exists a constant $r_0>0$ (independent of all the parameters
  above) such that the lower bound in \eqref{E:approach}
  holds for all $(z,w)\in \OM_F\cap\{(z,w):\im w < r_0\}$.
\end{enumerate} 
\end{theorem}

A further piece of notation: we shall
abbreviate $\met{\OM_F}(p;\,\xi, \xi)$\,---\,i.e., the Bergman metric for $\OM_F$
at $(p, \xi)$, which gives the \textbf{square} of the Bergman norm of
$\xi\in T_p^{1,0}\OM_F$\,---\,as $\met{F}(p;\,\xi)$. Our next theorem provides
estimates for the Bergman metric of $\OM_F$ as one approaches
$(0,0)\in \OM_F$.

\begin{theorem}\label{T:metric}
Let $\OM_F$ be the domain in $\CC$ described by Theorem~\ref{T:kernel}. Identify
$T^{1,0}\OM_F$ with $\OM_F\times\CC$ via the identification
$\xi = \xi_1\big(\!\left.\partial/\partial z\right|_p\big) + 
\xi_2\big(\!\left.\partial/\partial w\right|_p\big)\!\leftrightarrow\!(p;\,\xi_1,\xi_2)\in \OM_F\times\CC$.
Then, there exists a constant $C_2 > 0$ and, for each $\alpha>0$ and
$N\in\zahl_+$, there exists a constant $\tau(\alpha, N)>0$ such that
\begin{align}
  (1/C_2)\big(\big(f^{-1}(\im w)\big)^{-2}|\xi_1|^2
  + |\im w|^{-2}|\xi_2|^2\big)\,&\leq\,\met{F}(z,w;\,\xi) \notag \\
    &\leq\,C_2\big(\big(f^{-1}(\im w)\big)^{-2}|\xi_1|^2
    + |\im w|^{-2}|\xi_2|^2\big) \label{E:approach_met}\\
    \forall (z,w;\,\xi) &\in \big(\apprh\cap\{(z,w): \im{w}<\tau(\alpha, N)\}\big)\times\CC, \notag
\end{align}
where $\apprh$ is the approach region introduced in Theorem~\ref{T:kernel}.
\end{theorem}

We emphasise: what makes optimal estimates in the infinite-type case\,---\,even with the
simplifying assumption $(\bullet)$\,---\,challenging is that there is no obvious prototype
that describes the behaviour of the function $F$ at $0\in \C$. In the finite-type
case, the ``right'' prototype for the $F$ in \eqref{E:local}
(and how this prototype changes as the point $p$ varies) is dictated by Taylor's theorem:
this is the basis of the diverse estimates derived in the papers cited above. In
contrast, due to the challenge just mentioned, there are very few works in the
infinite-type case: see, for instance, \cite{kimLee:abBkaicitpd02, bharali:gBknitp11,
lee:aeBktdit13}. The set-up in \cite{kimLee:abBkaicitpd02, bharali:gBknitp11} is the
closest to that of Theorems~\ref{T:kernel} and~\ref{T:metric}. Theorem~\ref{T:kernel}
subsumes the main result in \cite{bharali:gBknitp11}. This is because (along with the features
of $F$ already discussed) our doubling condition on $\lamf$ is more permissive than the
control on $\lamf$ required in \cite{bharali:gBknitp11} (see Section~\ref{SS:subsumes}
for details). In \cite{kimLee:abBkaicitpd02}, the domains
$\OM_F$ are required to satisfy the following conditions (with $f$, as in
Theorem~\ref{T:kernel}, such that $f(|z|) = F(z)$ $\forall z\in \C$): 
\begin{itemize}[leftmargin=22pt]
  \item $f^{\prime\prime}(x) > 0$ $\forall x > 0$; and
  \item $\lamf$ extends smoothly to $x=0$ and vanishes to finite order at $0$.
\end{itemize}
The second condition does not allow Kim--Lee to study in \cite{kimLee:abBkaicitpd02}
such $\OM_F$ that are either ``mildly infinite-type'' or very flat at $(0,0)$. The conditions stated in
Theorem~\ref{T:kernel} \emph{do} allow us to analyse $\OM_F$ of the latter kind: an assertion that
will be clearer through the examples in Section~\ref{S:examples}.
\smallskip

Let us recall what is meant by vanishing to infinite order at $0$. In the context
of the domains $\OM_F$, we mean that the function $f$ is of class $\smoo^\infty([0,\infty))$,
and $f^{(n)}(0) = 0$, $\lim_{x\to 0^+}f(x)/x^n = 0$ for every $n\in \nat$.
\smallskip

A few analytic and geometric preliminaries are needed before the proofs of our main theorems
can be given. It might be helpful to get a sense of the key ideas of our proof. A discussion of our
method, plus the role of the localisation principle in \cite{chenKamimotoOhsawa:bBki03}
mentioned above, are presented in Section~\ref{S:BergPrelim}. Section~\ref{S:analLemmas}
is devoted to essential quantitative lemmas. The proofs of the main results will be presented
in Sections~\ref{S:proof_kernel} and~\ref{S:proof_metric}.
\smallskip

\section{Examples}\label{S:examples}
This section is devoted to presenting examples of domains of the form $\OM_F$ that
satisfy the conditions of Theorems~\ref{T:kernel} and~\ref{T:metric}. They are
such that the point $(0,0)\in \bdy\OM_F$ is a point of infinite type, but
$\bdy\OM_F$ will\,---\,as we shall see\,---\,be flat to varying degrees in these examples.
\smallskip

Let $F$ and $f$ be as in Theorem~\ref{T:kernel}. Since $F$ is assumed to
be radial and subharmonic, it is useful to recall the expression for the Laplacian on $\C$
in polar coordinates:
\[
  \laplc\,:=\,\pd{{}}{r^2}{2} + \frac{1}{r}\pd{{}}{r}{{}}
  		+ \frac{1}{r^2}\pd{{}}{\theta^2}{2},
\]
where we write $z = re^{i\theta}$. In view of the assumption $F(re^{i\theta}) = f(r)$
$\forall r>0$ and $\forall\theta\in \R$, we immediately have the following:

\begin{lemma}\label{L:radial_subh}
Let $F: \C\lrarw \R$ be a radial function and let $f: [0, \infty)\lrarw \R$ be such that $F(re^{i\theta}) = f(r)$
$\forall r\geq0$ and $\forall\theta\in \R$, where $f\in \smoo^2\big([0, \infty)\big)$. Suppose
$f^{(n)}(x) = o(x^{2-n})$ as $x\to 0^+$ for $n = 1, 2$. Furthermore, if
\[
  f^{\prime\prime}(x) + x^{-1}f^\prime(x)\,\geq\,0 \; \; \forall x > 0,
\]
then $F$ is subharmonic on $\C$.
\end{lemma}

Our first example features the familiar functions $f(x) = e^{-1/x^p}$, $x>0$ (where $p>0$),
which vanish to infinite order at $x=0$.

\begin{example}\label{Ex:mildly_inf_type}
A class of domains $\OM_F$ satisfying the conditions of
Theorems~\ref{T:kernel} and~\ref{T:metric}, which includes domains for which $\bdy\OM_F$ is mildly
infinite-type at $(0,0)$.
\end{example}

\noindent{Consider the function $f : [0,\infty)\lrarw [0, \infty)$ described by the following
conditions:
\begin{enumerate}[leftmargin=25pt]
  \item[$(a)$] Fixing a constant $p>0$,
  \[
    f(x)\,:=\,\begin{cases}
    			e^{-1/x^p}, &\text{if $0<x<1/2$},\\
    			0, &\text{if $x=0$}.
    			\end{cases}
  \]
  \item[$(b)$] $f|_{[1/2, \infty)}$ is so defined that $f|_{(0, \infty)}$ is of class 
  $\smoo^\infty$ and strictly increasing, the function $F: \C\lrarw [0, \infty)$ given by $F(z) := f(|z|)$
  is subharmonic on $\C$, and $\lim_{x\to \infty}f(x) = +\infty$.
\end{enumerate}
We shall soon see why it is possible to satisfy all of the conditions listed in $(b)$. But first:
notice that if $p\not\in \zahl_+$, then $\lamf$ does not extend smoothly to $x=0$\,---\,which
places Example~\ref{Ex:mildly_inf_type} outside the realm considered by
Kim--Lee in \cite{kimLee:abBkaicitpd02}. We shall see that the conditions of
Theorems~\ref{T:kernel} and~\ref{T:metric} are satisfied for $p$ arbitrarily close to $0$.
With $f$ as above, when $p\ll 1$ we say that
$\bdy\OM_F$ is \emph{mildly infinite-type at $(0,0)$}.
\smallskip

Let us write
\[
  \phi_p(x)\,:=\,e^{-1/x^p} \; \; \forall x\in (0,1).
\]
We compute:
\begin{align*}
  \phi_p^\prime(x)\,&=\,px^{-(p+1)}e^{-1/x^p}\,\big(>\,0
  					\; \; \forall x\in (0,1)\big),\\
  \phi_p^{\prime\prime}(x)\,&=\,e^{-1/x^p}\bcdot
  \big(p^2x^{-2(p+1)} - p(p+1)x^{-(p+2)}\big).
\end{align*}
Clearly
\begin{equation}\label{E:strict_pos}
  \phi_p^{\prime\prime}(x) + x^{-1}\phi_p^\prime(x)\,>\,0 \; \; \forall x: 0<x<1.
\end{equation}
It is well-known (we shall skip calculating further higher-order derivatives) that
$\phi_p$ extends to $[0,1)$ to belong to $\smoo^\infty([0,1))$ and  
vanishes to infinite order at $0$. In
view of \eqref{E:strict_pos}
and Lemma~\ref{L:radial_subh}, we conclude that the function $\Phi_p(z) := \phi_p(|z|)$
is subharmonic on the open unit disc.
\smallskip

It is easy to extend $\phi_p|_{(0,1/2]}$ to a $\smoo^\infty$ function on $(0, \infty)$
by matching the $n$-th derivative at $1/2$, of some smooth function on $[1/2, \infty)$, with
${\phi_p}^{\!\!(n)}(1/2)$, $n\in \nat$. If we call this extension $f$ and let $F$ be as
given by $(b)$, then, as $\laplc\Phi_p$ is \emph{strictly}
positive on the circle $\{z\in \C: |z| = 1/2\}$ (see \eqref{E:strict_pos} above), we
can also arrange for $\laplc F > 0$ on $\{z\in \C: |z|\geq 1/2\}$ and, indeed, for
$f$ to have all the properties stated in $(b)$ above. 
\smallskip

To complete the discussion of Example~\ref{Ex:mildly_inf_type}, we must show that $\lamf$
satisfies a doubling condition. Here,
$\lamf(x) = x^p$ $\forall x\in [0,1/2]$. Hence, if we fix some $\sigma\geq 2^{1/p}$\,($>1$),
then we have
\[
  2\lamf(x)\,\leq\,\lamf(\sigma x) \; \; \forall x\in [0, 1/2\sigma].
\]
Hence, $\OM_F$ satisfies the conditions of
Theorems~\ref{T:kernel} and~\ref{T:metric}.
\hfill $\blacktriangleleft$}
\medskip

Our next example is an illustration of a domain $\OM_F$ where
$\bdy\OM_F$ may be described to be extremely flat at $(0,0)$. There are some
commonalities in the methods used in \cite{kimLee:abBkaicitpd02} and in this
paper, which we shall elaborate on in Section~\ref{S:BergPrelim}. The key difference
between the two approaches is that Kim--Lee rely on scaling methods
in \cite{kimLee:abBkaicitpd02} to complete their proofs. Although we seek slightly
different conclusions from those
in \cite{kimLee:abBkaicitpd02}, if we were to rely on scaling methods, then we
would need a non-trivial Taylor approximation of $\lamf(x)$ around $x=0$, as is
the case in \cite{kimLee:abBkaicitpd02}. This is just not available for the $\lamf$
in Example~\ref{Ex:very_flat}, which is the relevance of this example.
 
\begin{example}\label{Ex:very_flat}
A domain $\OM_F$ satisfying the conditions of
Theorems~\ref{T:kernel} and~\ref{T:metric} such that $\bdy\OM_F$ is extremely
flat at $(0,0)$.
\end{example}

\noindent{Let $\psi$ be the function $\left.\phi_1\right|_{(0,1/2)}$, where
$\phi_1$ is as introduced in Example~\ref{Ex:mildly_inf_type}.
Now consider the function $f : [0,\infty)\lrarw [0, \infty)$ described by the following
conditions:
\begin{enumerate}[leftmargin=25pt]
  \item[$(a^\prime)$] With $\psi$ as above,
  \[
    f(x)\,:=\,\begin{cases}
    			e^{-1/\psi(x)}, &\text{if $0<x<1/2$},\\
    			0, &\text{if $x=0$}.
    			\end{cases}
  \]
  \item[$(b^\prime)$] $f|_{[1/2, \infty)}$ is so defined that $f|_{(0, \infty)}$ is of class 
  $\smoo^\infty$ and strictly increasing, the function $F: \C\lrarw [0, \infty)$ given by $F(z) := f(|z|)$
  is subharmonic on $\C$, and $\lim_{x\to \infty}f(x) = +\infty$.
\end{enumerate}
As in the discussion of Example~\ref{Ex:mildly_inf_type}, let us write
\[
  \phi(x)\,:=\,e^{-1/\psi(x)} \; \; \forall x\in (0,1).
\]
We shall omit the essentially elementary calculations showing that
$\phi$ extends to $[0,1)$ to belong to $\smoo^\infty([0,1))$ and
vanishes to infinite order at $0$. Just to indicate
the calculations needed: the last statement follows from the Fa{\'a} di Bruno formula
for the higher derivatives of the composition of two univariate functions (see 
\cite[Chapter~1]{krantzParks:praf02}, for instance) and the fact that
\[
  \lim_{x\to 0^+}e^{n/x}\sqrt{e^{-1/\psi(x)}}\,=\,\lim_{x\to 0^+}e^{n/x}
  \exp\big(\!-\!2^{-1}e^{1/x}\big)\,=\,\lim_{y\to 0^+}\frac{e^{-1/2y}}{y^n}\,=\,0
\]
for every $n\in \zahl_+$.
\smallskip 

However, it is useful to calculate couple of derivatives:
\begin{align*}
  \phi^\prime(x)\,&=\,x^{-2}e^{1/x}\exp\big(\!-\!e^{1/x}\big)\,\big(>\,0
  					\; \; \forall x\in (0,1)\big),\\
  \phi^{\prime\prime}(x)\,&=\,\exp\big(\!-\!e^{1/x}\big)\bcdot
  \big(x^{-4}e^{2/x} - x^{-4}e^{1/x} - 2x^{-3}e^{1/x}\big).
\end{align*}
Clearly
\begin{equation}\label{E:strict_pos_again}
  \phi^{\prime\prime}(x) + x^{-1}\phi^\prime(x)\,>\,0 \; \; \forall x: 0<x<1.
\end{equation}
By \eqref{E:strict_pos_again}
and Lemma~\ref{L:radial_subh}, we deduce that $\Phi(z) := \phi(|z|)$ (with
$\Phi(0) := 0$) is subharmonic on the open unit disc. By arguments analogous to those for
Example~\ref{Ex:mildly_inf_type}, it is easy to extend $\phi|_{(0,1/2]}$ to a
$\smoo^\infty$ function $f$ defined on $(0, \infty)$ so that $f$ has
all the properties listed in $(b^\prime)$. 
\smallskip

To complete the discussion of Example~\ref{Ex:very_flat}, we must show that $\lamf$
satisfies a doubling condition. Here,
$\lamf(x) = \psi(x)$ $\forall x\in [0,1/2]$. Fix an $\sigma$ such that
\[
  (\sigma-1)(\log{2})^{-1}\,\geq\,1/2.
\]
Then, whenever $0<\sigma{x}\leq 1/2$, we have
\[
 \sigma{x}\,\leq\,(\sigma-1)(\log{2})^{-1} \;\Rightarrow \;
  \log{2} - \frac{1}{x}\,\leq\,-\frac{1}{\sigma{x}}, 
\]
which implies that $2\psi(x)\leq \psi(\sigma{x})$ whenever $0\leq\sigma{x}\leq 1/2$.
Therefore, $\OM_F$ satisfies the
conditions of Theorems~\ref{T:kernel} and~\ref{T:metric}.
\hfill $\blacktriangleleft$}
\medskip

\section{Preliminaries}\label{S:BergPrelim}
This section is devoted to introducing the key ideas underlying the proofs in this paper. To this end, we
begin by introducing some of the notation that we shall frequently use.

\subsection{Common notations}
We fix the following notation.
\begin{enumerate}[leftmargin=25pt]
  \item $\D$ will denote the open unit disc in $\C$ with centre at $0$, 
  while $D(a, r)$ will denote the open disc in $\C$ with radius $r > 0$ and centre $a$.
  \vspace{0.65mm}
  
  \item For $\xi \in\CC$ (or, in general, in $\C^n$), $\|\xi\|$ will denote the Euclidean norm.
  Given points $z, w \in \C^n$, we shall commit a mild abuse of notation by not distinguishing
  between points and tangent vectors, and denote the Euclidean distance between them as
  $\|z - w\|$.
\end{enumerate}

\subsection{On the lower bounds presented in Theorems~\ref{T:kernel} and~\ref{T:metric}}
We now present an overview of how we shall derive the lower bounds given
by Theorems~\ref{T:kernel} and~\ref{T:metric}, which are the non-trivial parts of
these results. Implicit in both these
theorems is the fact that $A^2(\OM_F)$ is non-trivial. This, and a lot
else, follows from the following localisation result. We are able to invoke this result owing to
the conclusions of part~$(a)$ of Theorem~\ref{T:kernel}.

\begin{result}[paraphrasing {\cite[Lemma~3.2]{chenKamimotoOhsawa:bBki03}} by
Chen--Kamimoto--Ohsawa]\label{R:key}
Let $\OM := \{(z,w)\in \C^n\times\C : \im{w} > \rho(z)\}$, where 
$\rho$ is a non-negative plurisubharmonic function such that $\rho(0)=0$
and $\lim_{\|z\|\to \infty}\rho(z) = +\infty$. Let $V\Subset U$ be two open neighbourhoods of 
$0\in \bdy\OM$. Then, there is a constant $\delta\equiv\delta(U,V)>0$ such that
\begin{align}
  K_{\OM}(z,w)\,&\geq\,\delta K_{\OM\cap U}(z,w) \; \; \forall (z,w)\in \OM\cap V,
  \label{E:locali_kern} \\
  \met{\OM}(z,w;\,\xi)\,&\geq\,\delta \met{\OM\cap U}(z,w;\,\xi) \; \;
  \forall (z,w;\,\xi)\in \big(\OM\cap V\big)\times \C^{n+1}.
  \label{E:locali_met}
\end{align}
\end{result}

This localisation result allows us to obtain lower bounds for the quantities of interest
by finding lower bounds for the respective quantities associated to $\OM_F\cap\Delta$, where
$\Delta$ is a well-chosen bidisc centered at $(0,0)\in \bdy\OM_F$. We shall obtain the
latter lower bounds by appealing to certain extremal problems\,---\,sometimes referred to as
the Bergman--Fuchs formulas\,---\,that give the values of the
Bergman kernel (evaluated on the diagonal), and of the Bergman metric, for bounded domains:
see \cite{bergman:kbvr33} by Bergman (also see \cite{fuchs:gmepairg37} by Fuchs).
\smallskip

In this paragraph,
$\OM$ will denote an arbitrary domain in $\C^2$ (we restrict ourselves to $\C^2$ to avoid having
to define further notation). One of the Bergman--Fuchs formulas is:
\begin{equation}\label{E:B-F_kern}
  K_{\OM}(z,w)\,=\,\sup\left\{\frac{|\varphi(z,w)|^2}{\|\varphi\|^2_{\leb{2}(\OM)}} :
					\varphi\in A^2(\OM)\right\} \; \; \forall (z,w)\in \OM.
\end{equation}
A related formula is known for $\met{\OM}$. To see this, we need the following auxiliary
quantity
\begin{multline}\label{E:B-F_aux}
  \jay_{\OM}(z,w;\,\xi)\,:=\,\inf\left\{\|\varphi\|^2_{\leb{2}(\OM)} : \varphi\in A^2(\OM), \ 
  						\varphi(z,w)=0 \text{ and }\right. \\
  						 \partial_z\varphi(z,w)\xi_1 + \partial_w\varphi(z,w)\xi_2 = 1\Big\},
  						 \; \; (z,w)\in \OM, \ \xi\in \CC\!\setminus\!\{0\}.
\end{multline}
The Bergman--Fuchs formula for $\met{\OM}$ is
\begin{equation}\label{E:B-F_met}
  \met{\OM}(z,w;\,\xi)\,=\,\frac{1}{K_{\OM}(z,w)\,\jay_{\OM}(z,w;\,\xi)}
  \; \; \forall (z,w;\,\xi)\in \OM\times(\CC\!\setminus\!\{0\}).
\end{equation}

How these formulas help in deriving  the lower bounds given by Theorems~\ref{T:kernel}
and~\ref{T:metric} is summarised as follows:
\begin{itemize}[leftmargin=22pt]
  \item \emph{Step 1:} We choose a suitable bidisc $\Delta$ centered at
  $(0,0)$ (determined just by $f$).
  To obtain a lower bound for $K_{\loc}(z,s+it)$, we just need
  to find a suitable function $\varphi_t\in A^2(\loc)$, $t>f(|z|)$, such that\,---\,owing
  to \eqref{E:B-F_kern}\,---\,$\|\varphi_t\|^2_{\leb{2}(\loc)}$ has an upper bound
  that induces the lower bound in \eqref{E:approach}.
  \vspace{1mm}
  
  \item \emph{Step~2:} The latter task reduces
  to estimating an integral over a region in $\R^4$ whose boundaries are determined by $f$.
  The doubling condition is used to break up this
  region of integration into sub-domains on which the relevant integral is easier to
  estimate to sufficient precision that we get the desired upper bound.
  \vspace{1mm}
  
  \item \emph{Step~3:} In view of \eqref{E:B-F_aux}, we must
  to find a suitable function $\widetilde{\varphi}_t$ belonging, this time, to
  the class $\left\{\varphi\in A^2(\loc) :
  \varphi(z,w)=0 \text{ and }
  \partial_z\varphi(z,w)\xi_1 + \partial_w\varphi(z,w)\xi_2 = 1\right\}$
  in order to deduce a lower bound for  $\met{\loc}(z,s+it;\,\xi)$. In view of
  \eqref{E:B-F_met}, we need to obtain an upper bound of a specific form 
  for $\|\widetilde{\varphi_t}\|^2_{\leb{2}(\loc)}$. A procedure analogous to that
  described in Step~2 applies in computing the latter upper bound.
\end{itemize}
The final estimates hinted at by the above summary 
lead to the lower
bounds that we want\,---\,i.e., for the domain $\OM_F$\,---\,by the use of Result~\ref{R:key}.
\smallskip

To conclude this section, we elaborate upon some comments made in Section~\ref{S:intro}
about our condition on $\lamf$ in comparison to \cite{bharali:gBknitp11}.

\subsection{Relation to the main result in \textbf{{\cite{bharali:gBknitp11}}}}\label{SS:subsumes}
We give a justification of the assertion in Section~\ref{S:intro} that 
Theorem~\ref{T:kernel} subsumes the main result in \cite{bharali:gBknitp11}.
Given our statements on the condition~$(*)$ in \cite{bharali:gBknitp11}, it
suffices to show that the condition imposed on $\lamf$ in \cite{bharali:gBknitp11} implies that
$\left.\lamf\right|_{[0, R]}$ satisfies a doubling condition for some $R>0$.
To this end, recall that for $F$ (and the associated $f$) as in
\cite[Theorem~1]{bharali:gBknitp11}, there exist constants $\eps_0 > 0$ and $B\geq 1$ such that
\begin{equation}\label{E:old_lamf_bounds}
  (1/B)\chi(x)\,\leq\,\lamf(x)\,\leq\,B\chi(x) \; \; \forall x\in[0,\eps_0],
\end{equation}
where $\chi\in \smoo([0,\eps_0])$ is an increasing function
such that $\chi^p$ is convex on $(0,\eps_0)$ for some
$p>0$. It follows that, setting $\nu:=\min\{m\in \nat: 2^m\geq p\}$, $\chi|_{[0,2^{-(\nu+1)}\eps_0]}$
satisfies a doubling condition. 
Write $R := 2^{-(\nu+1)}\eps_0$ and let $\sigma > 1$ be a 
doubling constant for $\chi|_{[0, R]}$. Let $N\in \zahl_+$ be such that
$2^N\geq 2B^2$. Then, by \eqref{E:old_lamf_bounds}
\[
  2\lamf(x)\,\leq\,2B\chi(x)\,\leq\,2^N(1/B)\chi(x)\,\leq\,(1/B)\chi(\sigma^Nx)\,\leq\,\lamf(\sigma^Nx)
  \; \; \forall x\in [0, R/\sigma^N].
\]
In view of the above discussion, it follows that Theorem~\ref{T:kernel} subsumes the
main result in \cite{bharali:gBknitp11}.
\smallskip

\section{Technical lemmas}\label{S:analLemmas}
We present some lemmas that play a supporting role in the proofs of
Theorems~\ref{T:kernel} and~\ref{T:metric}.
\smallskip

\begin{lemma}\label{L:diff_of_squares}
Let $f: [0, \infty)\lrarw \R$ be a strictly increasing function satisfying $f(0)=0$. Let
$\lamf|_{[0,R]}$ satisfy a doubling condition for some $R\in \big(0, f^{-1}(1)\big)$. Write
$G_f := \lamf^{-1}$. There exist constants $T, C'>0$ such that
\[
  0\,\leq\,G_f(2t)^{2n}-G_f(t)^{2n}\,\leq\,C'G_f(t)^{2n} \; \; \forall t\in [0,T],
\]
$n=1,2$.
\end{lemma}
\begin{proof}
Let $\sigma > 1$ be a doubling constant for $\lamf$. Write $T:= \lamf(R/\sigma)$.
Since, by the doubling condition,
\begin{equation}\label{E:doub}
  2\lamf(x)\,\leq\,\lamf(\sigma x\big) \; \; \forall x\in [0, R/\sigma],
\end{equation}
it follows that
\[
  G_f\big(2\lamf(x)\big)\,\leq\,\sigma x \; \; \forall x\in [0, R/\sigma].
\]
This inequality holds on the interval stated since, by 
\eqref{E:doub}, $2\lamf(x)\in {\sf dom}(G_f)$ $\forall x\in [0, R/\sigma]$.
Parametrising the latter interval by $G_f: [0, T]\lrarw [0, R/\sigma]$, we
can take $x = G_f(t)$, $t\in [0, T]$, in the last inequality to
get
\[
  G_f(2t)\,\leq\,\sigma G_f(t) \; \; \forall t\in [0, T].
\]
This implies:
\begin{equation}\label{E:penultimate}
  G_f(2t)-G_f(t)\,\leq\,(\sigma-1)G_f(t) \; \; \forall t\in [0, T].
\end{equation}
Since
\[
  G_f(2\,\bcdot)^2-\geef{2}\,=\,\big(G_f(2\,\bcdot)-G_f\big)^2 + 2\big(G_f(2\,\bcdot)-G_f\big)G_f\,,
\]
and
\begin{align*}
  G_f(2\,\bcdot)^4-\geef{4}\,&=\,\big(G_f(2\,\bcdot)^2-\geef{2}\,\big)
  									\big(G_f(2\,\bcdot)-G_f\big)^2 \\
  						&\quad + 2\big(G_f(2\,\bcdot)^2-\geef{2}\,\big)
							\big(G_f(2\,\bcdot)-G_f\big)G_f
							+ 2\big(G_f(2\,\bcdot)^2-\geef{2}\,\big)\geef{2}\,,
\end{align*}
we can find an appropriate constant $C'>0$ so that
the desired conclusion follows from \eqref{E:penultimate}.
\end{proof}

The aim of our next two lemmas is to estimate the norms of certain functions in $A^2(\loc)$, where
$\Delta$ is an appropriately chosen bidisc, from which we shall build candidates for such
functions as can be used in the argument sketched
in Steps~1--3 in Section~\ref{S:BergPrelim}.

\begin{lemma}\label{L:L_2_norm}
Let $f\in \smoo^\infty\big([0, \infty)\big)$ be a strictly increasing function that vanishes to infinite
order at $0$ and let
$\lamf$ satisfy the condition stated in Lemma~\ref{L:diff_of_squares}.
Let $F$ and $\OM_F$ be determined by $f$ as described in Section~\ref{S:intro}.
Write $a:=\min\{f^{-1}(1), 1\}$ and write $\Delta:=D(0,a)\times\D$. There
exist constants $C^*, r_0>0$ such that, for any $n\in \{0,1\}$, $\alpha, t>0$, $\beta>1$
and $z\in \C$, if we write
\[
  \psi(\zt,w; \alpha, \beta, n, t, z)\,:=\,\frac{|z|^\alpha\,t^\beta\,\zt^n}{(w+it)^2}
  \; \; \forall (\zt, w)\in \loc,
\]
then
\begin{align}
  \|\psi(\bcdot\,;\alpha, \beta, n, t, z)\|^2_{\leb{2}(\loc)}\,\leq\,C^*t^{2(\beta-1)}
  									&\big(f^{-1}(t)\big)^{2(\alpha+n+1)} \notag \\
  &\forall (z,t): (z,it)\in \loc \text{ and } t<r_0.  \label{E:key_norm_bd}
\end{align}
\end{lemma}
\begin{proof}
Let us write $w=u+iv$ and abbreviate $\psi(\bcdot\,;\alpha, \beta, n, t, z)$ as $\psi$.
We leave it to the reader to verify that we can apply Fubini's theorem
wherever necessary in the following computation:
\begin{align}
  \|\psi\|^2_{\leb{2}(\loc)}\,&=\,\int_{|\zt|\leq a}\ \int_{-\sqrt{1-F^2(\zt)}}^{\sqrt{1-F^2(\zt)}}
  \ \int_{F(\zt)}^{\sqrt{1-u^2}}
  \frac{|z|^{2\alpha}\,t^{2\beta}\,|\zt|^{2n}}{|u+i(v+t)|^4}dv\,du\,dA(\zt) \notag \\
  &\leq\,|z|^{2\alpha}\,t^{2\beta}\lint{|\zt|\leq a}{{}} \ \lint{F(\zt)}{\infty} \ \lint{-1}{1}(v+t)^{-4}
   \left(1+\Big(\frac{u}{v+t}\Big)^2\right)^{-2}|\zt|^{2n}\,du\,dv\,dA(\zt) \notag \\
  &\leq\,|z|^{2\alpha}\,t^{2\beta}\Bigg(\int_{\mathbb{R}}\frac{ds}{(1+s^2)^2}\Bigg)
   \lint{|\zt|\leq a}{{}} \ \lint{F(\zt)}{\infty} \ (v+t)^{-3}|\zt|^{2n}\,dv\,dA(\zt) \notag \\
  &=\,C|z|^{2\alpha}\,t^{2\beta}\int_{0}^{a}\frac{r^{2n+1}}{(t+f(r))^2}\,dr\,, \label{E:f-integ}
\end{align}
where $C>0$ is a universal constant.
\smallskip

In the remainder of this argument, $B>0$ will denote a constant
whose value is independent of the variables involved, whose actual value is not of interest, and
which may change from line to line.
For any $y>0$, set $R_{y}:=f^{-1}(y)$. Write $G_f = \lamf^{-1}$.
By definition, we get
\begin{equation}\label{E:R_sqrt}
  R_{\sqrt{t}}\,=\,G_f\left(\frac{2}{\log(1/t)}\right), \; \; 0<t<1.
\end{equation}
We now break up the interval of integration of the integral in \eqref{E:f-integ}. For simplicity
of notation, we shall initially consider all $t$ such that $0<t<f(a)^2$, to get:
\begin{align}
  \int_{0}^{a}\frac{r^{2n+1}}{(t+f(r))^2}\,dr\,&=\,\left(\int_0^{R_t} + \ \int_{R_t}^{R_{\sqrt{t}}} + \
  \int_{R_{\sqrt{t}}}^a\frac{r^{2n+1}}{(t+f(r))^2}\,dr\right) \notag \\
  &\leq\,\int_0^{R_t}\frac{r^{2n+1}}{t^2}\,dr +
   \left(\int_{R_t}^{R_{\sqrt{t}}} + \ \int_{R_{\sqrt{t}}}^a\,\frac{r^{2n+1}}{4tf(r)}\,dr\right) \notag  \\
  &\leq\,B\Big(t^{-2}R_t^{2n+2} + \frac{1}{tf(R_{\sqrt{t}})}\,a^{2n+2}\Big)
   + \int_{R_t}^{R_{\sqrt{t}}}\frac{r^{2n+1}}{4tf(r)}\,dr \notag \\
  &\leq\,B\big(t^{-2}R_t^{2n+2} + t^{-3/2}\big)
   + \int_{R_t}^{R_{\sqrt{t}}}\frac{r^{2n+1}}{4tf(r)}\,dr \notag \\
  &\leq\,B\big(t^{-2}R_t^{2n+2} + t^{-3/2}\big) \notag \\
  &\qquad\qquad + \frac{t^{-2}}{4}\Bigg[G_f\!\left(\frac{2}{\log(1/t)}\right)^{2n+2}
   - G_f\!\left(\frac{1}{\log(1/t)}\right)^{2n+2}\Bigg].    \label{E:modulo_middle}
\end{align} 
In the above calculation, the third inequality follows from the fact that, by definition, $a\leq 1$, while
the estimate for the middle integral draws upon \eqref{E:R_sqrt}.
\smallskip

We shall now apply Lemma~\ref{L:diff_of_squares} to the expression in brackets
in \eqref{E:modulo_middle}. Let 
$T>0$ and $C'>0$ be as given by that lemma. At this stage,
let us fix $t$ to be in $(0, \min\{f(a)^2, e^{-1/T}\})$. By Lemma~\ref{L:diff_of_squares} and
\eqref{E:modulo_middle}:
\begin{align}
  \int_{0}^{a}\frac{r^{2n+1}}{(t+f(r))^2}\,dr\,&\leq\,B(t^{-2}R_t^{2n+2} + t^{-3/2})
  							+ \frac{C'}{4}t^{-2}G_f\!\left(\frac{1}{\log(1/t)}\right)^{2n+2} \notag \\
											&\leq\,B\big(t^{-2}R_t^{2n+2} + t^{-3/2}\big). \label{E:integ2}
\end{align}											
By the hypothesis that $f$ vanishes to infinite order at $0$, it follows that for any
$p, q>0$,
\begin{equation}\label{E:slow_decay}
  \frac{\big(f^{-1}(t)\big)^p}{t^q}\lrarw \infty \text{ as $t\to 0^+$}.
\end{equation}
Thus, there is a constant $c>0$ such that
\[
  R_t^{2n+2}\,\geq\,t^{1/2} \; \; \forall t\in (0,c).
\]
Set $r_0:=\min\{f(a)^2, e^{-1/T}, c\}$. Then, from the above inequality, \eqref{E:integ2} and
\eqref{E:f-integ}, we get
\[
  \|\psi\|^2_{\leb{2}(\loc)}\,\leq\,C\bcdot B|z|^{2\alpha}\,t^{2\beta-2}\big(f^{-1}(t)\big)^{2n+2}
  \; \; \forall t\in (0,r_0) \text{ and $n=0,1$}.
\]
Recall that if $(z,it)\in \OM_F$, then $t>F(z)=f(|z|)$. From this and the previous inequality
(we set $C^*:=C\bcdot B$), the lemma follows.
\end{proof}

A part of the proof of Theorem~\ref{T:metric} requires estimates for the norms
of certain functions in $A^2(\loc)$ that are not addressed by Lemma~\ref{L:L_2_norm}. Thus we
need:
 
\begin{lemma}\label{L:L_2_norm_alt}
Let $f\in \smoo^\infty\big([0, \infty)\big)$, $a>0$, and $\OM_F$, $\Delta\subset \C^2$
be exactly as in Lemma~\ref{L:L_2_norm}. There
exist constants $C^*, r_0>0$ such that, for any $n\in \{0,1\}$ and $t>0$, if we write
\[
  \phi(\zt,w; n,t) := \frac{t^3w^n}{(w+it)^{3}} \; \; \forall (\zt, w)\in \loc,
\]
then
\begin{equation}
  \|\phi(\bcdot\,;n, t)\|^2_{\leb{2}(\loc)}\,\leq\,C^*t^{2+2n}\big(f^{-1}(t)\big)^{2} \; \;
  \forall t\in (0, r_0).  \label{E:key_norm_bd_alt}
\end{equation}
\end{lemma}
\begin{proof}
As in the proof of Lemma~\ref{L:L_2_norm}, we
write $w=u+iv$ and, for $n\in \{0,1\}$, compute:
\begin{align}
  \|\phi(\bcdot\,;n, t)\|^2_{\leb{2}(\loc)}\,&=\,\int_{|\zt|\leq a}\ \int_{-\sqrt{1-F^2(\zt)}}^{\sqrt{1-F^2(\zt)}}
  \ \int_{F(\zt)}^{\sqrt{1-u^2}}
  \frac{t^6(u^2 + v^2)^n}{|u+i(v+t)|^6}dv\,du\,dA(\zt) \notag \\
  &\leq\,t^6\!\!\lint{|\zt|\leq a}{{}} \ \lint{F(\zt)}{\infty} \ \lint{-1}{1}
     \frac{u^{2n}}{|u+i(v+t)|^6} +  \frac{v^{2n}}{|u+i(v+t)|^6}\,du\,dv\,dA(\zt) \notag \\
  &\equiv t^6(I_1 + I_2). \label{E:f-integ_crude}
\end{align}

Next, we estimate:
\begin{align}
  I_1\,&=\,\lint{|\zt|\leq a}{{}} \ \lint{F(\zt)}{\infty} \ \lint{-1}{1}(v+t)^{2n-6}
   \frac{\big(u/(v+t)\big)^{2n}}{\left(1 + \big(u/(v+t)\big)^2\right)^{\!3}}\,du\,dv\,dA(\zt) \notag \\
  &\leq\,\Bigg(\int_{-1}^{1}\frac{s^{2n}}{(1+s^2)^3}\,ds\Bigg)
   \int_{|\zt|\leq a} \ \int_{F(\zt)}^{\infty} \ (v+t)^{2n-5}\,dv\,dA(\zt) \notag \\
  &=\,C\int_{0}^{a}\frac{r}{(t+f(r))^{4-2n}}\,dr\,, \label{E:f-integ_finer1}
\end{align}
and, analogously:
\begin{align}
  I_2\,&=\,\lint{|\zt|\leq a}{{}} \ \lint{F(\zt)}{\infty} \ \lint{-1}{1}(v+t)^{2n-6}
   \frac{\big(v/(v+t)\big)^{2n}}{\left(1 + \big(u/(v+t)\big)^2\right)^{\!3}}\,du\,dv\,dA(\zt) \notag \\
  &\leq\,\Bigg(\int_{-1}^{1}\frac{ds}{(1+s^2)^3}\Bigg)
   \int_{|\zt|\leq a} \ \int_{F(\zt)}^{\infty} \ (v+t)^{2n-5}\,dv\,dA(\zt) \notag \\
  &=\,C\int_{0}^{a}\frac{r}{(t+f(r))^{4-2n}}\,dr\,, \label{E:f-integ_finer2}
\end{align}
where, in both estimates above, $C>0$ is a constant independent of $t$ and $n$.
\smallskip

For any $y>0$, define $R_{y}:=f^{-1}(y)$. 
In what follows, $B>0$ will denote a constant whose value is independent of the variables
involved, and which may change from line to line.
Since the intermediate inequalities leading to \eqref{E:integ2_alt} are \emph{completely analogous}
to those in the proof of Lemma~\ref{L:L_2_norm}, we shall be brief. From
\eqref{E:f-integ_finer1} and \eqref{E:f-integ_finer2}:
\begin{align}
  I_1+I_2\,&\leq\,B\left(\int_0^{R_t} + \ \int_{R_t}^{R_{\sqrt{t}}} + \
  \int_{R_{\sqrt{t}}}^a\frac{r}{(t+f(r))^{4-2n}}\,dr\right) \notag \\
  &\leq\,B\int_0^{R_t}\frac{r}{t^{4-2n}}\,dr +
   B\left(\int_{R_t}^{R_{\sqrt{t}}} + \ \int_{R_{\sqrt{t}}}^a\,\frac{r}{4\big(tf(r)\big)^{2-n}}\,dr\right) \notag  \\
  &\leq\,B\big(t^{2n-4}R_t^{2} + t^{3(n-2)/2}\big), \label{E:integ2_alt}
\end{align}											
provided $t\in (0, \min\{f(a)^2, e^{-1/T}\})$, where $T>0$ is as given by
Lemma~\ref{L:diff_of_squares}. The justification of the last inequality is,
essentially, the argument leading to the estimate \eqref{E:modulo_middle} above.
\smallskip

Since $f$ vanishes to infinite order at $0$, we can argue exactly
as in the previous proof to obtain a constant $c>0$ so that
$R_t^2\geq t^{1-(n/2)}$ whenever $t\in (0, c)$ (recall: $n\in \{0,1\}$).
Set $r_0:=\min\{f(a)^2, e^{-1/T}, c\}$.
Then, from the last inequality, \eqref{E:f-integ_crude} and 
\eqref{E:integ2_alt}, the estimate \eqref{E:key_norm_bd_alt} follows.
\end{proof}

\section{The proof of Theorem~\ref{T:kernel}}\label{S:proof_kernel}
The proof of one half of part~$(a)$ is, essentially, the proof of
\cite[Lemma~3.1]{bharali:gBknitp11}. We reproduce it with the aim of providing, for clarity,
a few details that were tacit in \cite{bharali:gBknitp11}. Suppose there exist
$r_1<r_2$, $r_1,r_2\in [0,\infty)$, such that $f(r_1)\geq f(r_2)$. As $f$ is continuous,
$f|_{[0, r_2]}$ attains its maximum in $[0, r_2]$ but, owing to our assumption,
there exists a point $r^*\in [0, r_2)$ such that
\[
  f(r^*)\,=\,\max\nolimits_{r\in [0, r_2]}f(r).
\]
Then, as $F$ is a radial function,
\[
  F(r^*)\,\geq\,F(z) \; \; \forall z\in D(0, r_2).
\]
Since $F$ is subharmonic,
the Maximum Principle implies that $F|_{D(0,r_2)}\equiv 0$. But this means
that the portion $\bdy\OM_F$ in $D(0, r_2)\times D(0, r_2)$ is Levi-flat, which is a
contradiction. Hence $f$ is strictly increasing. In particular, $F$ is non-constant. Thus, by
Liouville's theorem for subharmonic functions, $F$ is unbounded. As $f$ is strictly increasing,
it follows that $\lim_{r\to \infty}f(r)=+\infty$.
\smallskip

Fix $\alpha>0$ and $N\in \zahl_+$. We shall first find a constant $r(\alpha, N)>0$ such that the
upper bound in \eqref{E:approach} holds on $\apprh\cap\{(z,w): \im{w}<r(\alpha, N)\}$.
By part~$(a)$, $f^{-1}$ is well-defined. Then, with $G_f$ as in Lemma~\ref{L:diff_of_squares}, we have
the expression
\begin{equation}\label{E:inverse}
  f^{-1}(t)\,=\,G_f\left(\frac{1}{\log(1/t)}\right), \quad 0<t<1
\end{equation}
(which we have tacitly used in the proof of Lemma~\ref{L:L_2_norm}). 
When $0<\rho\leq 1/2$, 
\[
  f^{-1}(t/2)\,=\,G_f\left(\frac{1}
  {\log{2}+\log(1/t)}\right)\,\geq\,G_f\left(\frac{1}{2\log(1/t)}\right) \;\; \forall t\in(0,\rho).
\]
Let $C'$ and $T$ be as given by Lemma~\ref{L:diff_of_squares}. By this lemma\,---\,shrinking
$\rho>0$ if necessary so that $1/\log(1/t)\in (0,T)$ whenever $t\in(0,\rho)$\,---\,we get
\begin{multline}\label{E:compare1}
  f^{-1}(t)-f^{-1}(t/2)\,\leq\,G_f(1/\log(t^{-1}))-
  G_f(1/2\log(t^{-1})) \\
  \leq\,C'\,G_f(1/2\log(t^{-1}))\,\leq\,C'\,f^{-1}(t/2) \; \;\forall t\in [0,\rho).
\end{multline}
Write $c:=(C'+1)^{-1}$. Since $f(x)$ vanishes to infinite order
at $x=0$, there exists a constant $r(\alpha, N)>0$ such that $r(\alpha, N)\leq \rho$ and
\begin{equation}\label{E:compare2}
  \alpha t^{1/N}\,<\,\frac{c}{2}\,f^{-1}(t) \; \; \forall t\in(0, r(\alpha, N)).
\end{equation}
From \eqref{E:compare1} and \eqref{E:compare2}, we see that
\[
  |z|+\frac{c}{2}\,f^{-1}(t)\,<\,f^{-1}(t/2) \; \; \forall z:0\leq|z|<\alpha t^{1/N}, \ 
  0<t<r(\alpha, N),
\]
whence the bidisc
\[
 \triangle(z,t)\,:=\,D\Big(z, \frac{c}{2}\,f^{-1}(t)\Big)\times D(it, t/2)\,\subset\,\OM_F \; \;
 \forall (z,it)\in \apprh\cap\{\im{w}<r(\alpha, N)\}.
\]

Observe that the translations $T_s:(z,w)\longmapsto (z,s+w)$, $s\in\R$, are all
automorphisms of $\OM_F$. Thus, by the transformation rule for the Bergman
kernel, and by monotonicity, we get
\begin{align}
  K_F(z,s+it)\,&=\,K_F(z,it)  \notag \\
  			&\leq\,K_{\triangle(z,t)}(z,it)\,=\,\frac{1}{{\rm vol}\big(\triangle(z,t)\big)}
				\; \; \forall (z,s+it)\in \apprh\cap\{\im{w}<r(\alpha, N)\}. \label{E:Berg_poly}
\end{align}
The last equality follows from the fact that $\triangle(z,t)$ is a Reinhardt domain
centered at $(z,it)$. Hence, we have found a $C_1>0$, which is independent of the choice
of $\alpha$ and $N$, such that  
\begin{equation}\label{E:1stHalf}
  K_F(z,w)\,\leq\,C_1(\im{w})^{-2}\big(f^{-1}(\im{w})\big)^{-2} \; \;
  \forall (z,w)\in\apprh\cap\{\im{w}<r(\alpha, N)\}
\end{equation}
(here $C_1=16/c^2\pi^2$), which establishes one portion of part~$(b)$. 
\smallskip

We shall now deduce the desired lower bound. Set $a:=\min\{f^{-1}(1),1\}$. In the remainder
of this proof, $\Delta$ will denote the bidisc $D(0, a)\times\D$. Once again, we draw upon
the fact that the translations $T_s:(z,w)\longmapsto (z,s+w)$, $s\in\R$, are
automorphisms of $\OM_F$, whence:
\begin{align}
  K_F(z,s+it)\,&=\,K_F(z,it) \; \; \forall (z,s+it)\in \OM_F \label{E:trans_inv}\\
  				&\geq\,\delta K_{\loc}(z,it) \; \; 
				\forall (z,it)\in \OM_F\cap\big(\tfrac{1}{2}\Delta\big).
				\label{E:loc-to-Delta}
\end{align}
The second inequality is a consequence of Result~\ref{R:key} applied to $\OM_F$,
taking $U=\Delta$ and $V=\frac{1}{2}\Delta$. Part~$(a)$ of the present theorem enables
the use of Result~\ref{R:key}.
\smallskip

Now, consider the functions
\[
  \phi_t(\zt,w)\,:=\,-4t^2/(w+it)^2 \; \; \forall (\zt,w)\in \loc,
\]
where $t>0$. In the notation of Lemma~\ref{L:L_2_norm}, $\phi_t = -4\psi(\bcdot\,;0,2,0,t,1)$.
Let $r_0>0$ be as given by Lemma~\ref{L:L_2_norm}. By construction, $\phi_t(\zt,it)=1$
$\forall t>0$. Thus, by the Bergman--Fuchs identity \eqref{E:B-F_kern} and
the estimate \eqref{E:key_norm_bd} applied to $\phi_t$\,(\,$=-4\psi(\bcdot\,;0,2,0,t,1)$, as
explained), we have
\begin{equation}\label{E:low_bound_provisional}
  K_{\loc}(z,it)\,\geq\,(C^*)^{-1}t^{-2}\big(f^{-1}(t)\big)^{-2} \; \;
  \forall (z,it)\in \loc \text{ and } t<r_0.
\end{equation}
Lowering the value of $r_0$, if necessary, we may assume that
$\OM_F\cap\{(z,w): \im{w}<r_0\}\subseteq \OM_F\cap\big(\tfrac{1}{2}\Delta\big)$.
Then, from \eqref{E:trans_inv}, \eqref{E:loc-to-Delta} and \eqref{E:low_bound_provisional},
we get
\begin{equation}\label{E:low_bound_final}
  K_F(z,w)\,\geq\,\delta(C^*)^{-1}(\im w)^{-2}\big(f^{-1}(\im w)\big)^{-2} \; \;
  \forall (z,w)\in \OM_F\cap\{(z,w): \im{w}<r_0\}.
\end{equation} 
This establishes part~$(c)$ of our theorem. We may assume that each $r(\alpha, N)\leq r_0$
without affecting the inequality \eqref{E:approach}.
Now, consider the constant $C_1$ introduced in \eqref{E:1stHalf}: raising
the value of $C_1$, if necessary, we obtain a $C_1>0$ such that
part~$(b)$ of our theorem follows from the last observation, \eqref{E:1stHalf}
and \eqref{E:low_bound_final}. \hfill $\Box$
\smallskip

\section{The proof of Theorem~\ref{T:metric}}\label{S:proof_metric}
Part~$(a)$ of Theorem~\ref{T:kernel} is relevant to this proof as well. It establishes that $f$ is
invertible. Also relevant is the argument in the second paragraph of the proof of
Theorem~\ref{T:kernel}. The conclusion of this argument is summarised by the following:
\medskip

\noindent{\textsc{Fact.} \emph{There exists a constant $c>0$ and, for each $\alpha>0$ and
$N\in \zahl_+$, there exists a constant $r(\alpha, N)>0$ such that whenever
$(z,it)\in \apprh\cap\{(z,w): \im{w}<r(\alpha, N)\}$,
the bidisc
\begin{equation}\label{E:polyD_incl}
  \triangle(z,t)\,:=\,D\Big(z, \frac{c}{2}\,f^{-1}(t)\Big)\times D(it, t/2)\,\subset\,\OM_F.
\end{equation}}}

By an argument analogous to the one in the proof of Theorem~\ref{T:kernel}\,---\,involving
the fact that $T_s:(z,w)\longmapsto (z,s+w)$ is an automorphism of $\OM_F$ for any
$s\in \R$\,---\,we have
\begin{equation}\label{E:shift_inv}
  \met{F}(z,s+it;\,\xi)\,=\,\met{F}(z,it;\,\xi) \; \; \forall (z,s+it;\,\xi)\in \OM_F\times \CC.
\end{equation}
By \eqref{E:polyD_incl}, and by the monotonicity property of the functional $\jay_{\OM}$
given by \eqref{E:B-F_aux}, we have:
\[
  \jay_{\triangle(z,t)}(z,it;\,\xi)\,\leq\,\jay_{\OM_F}(z,it;\,\xi) \; \; 
  \forall (z,it;\,\xi)\in \big(\apprh\cap\{(z,w): \im{w}<r(\alpha, N)\}\big)\times(\CC\!\setminus\!\{0\}).
\]
We know that $K_{\triangle(z,t)}(z,it) = 16(\pi c)^{-2}\big(f^{-1}(t)\big)^{-2}t^{-2}$;
see \eqref{E:Berg_poly}. It is a standard result (or one may compute from the last formula)
that
\[
 \met{\triangle(z,it)}(z,it;\,\xi)\,=\,8\big(c^{-2}
						 \big(f^{-1}(t)\big)^{-2}|\xi_1|^2 + t^{-2}|\xi_2|^2\big) \; \; \forall \xi\in \CC.
\]
From these formulas and \eqref{E:B-F_met}, we get an \emph{exact} expression for
$\jay_{\triangle(z,t)}(z,it;\,\xi)$. We combine this with monotonicity of $\jay_{\OM}$: then,
\eqref{E:shift_inv}, the Bergman--Fuchs formula for $\met{F}$,
and the lower bound in \eqref{E:approach} imply that there exists a constant
$C_2 > 0$ (independent of $\alpha$ and $N$) such that:  
\begin{align*}
  \met{F}(z,s+it;\,\xi)\,&\leq\,1/\big(\jay_{\triangle(z,t)}(z,it;\,\xi)\,K_F(z, it)\big) \\
  						 &=\,16(\pi c)^{-2}\big(f^{-1}(t)\big)^{-2}t^{-2}\big(8c^{-2}
						 \big(f^{-1}(t)\big)^{-2}|\xi_1|^2 + 8t^{-2}|\xi_2|^2\big)\,K_F(z,it)^{-1} \\
						 &\leq\,C_2\big(\big(f^{-1}(t)\big)^{-2}|\xi_1|^2 + |t|^{-2}|\xi_2|^2\big) \\
						 &\qquad\qquad
						 	\forall (z,s+it;\,\xi)\in 
							\big(\apprh\cap\{(z,w): \im{w}<r(\alpha, N)\}\big)\times(\CC\!\setminus\!\{0\}).
\end{align*}
Hence, we have found a $C_2>0$, which is independent of the choice
of $\alpha$ and $N$, such that
\begin{align}
  \met{F}(z,w;\,\xi)
  &\leq\,C_2\big(\big(f^{-1}(\im w)\big)^{-2}|\xi_1|^2
  + |\im w|^{-2}|\xi_2|^2\big) \notag \\
  &\qquad\qquad
  \forall (z,w;\,\xi) \in \big(\apprh\cap\{(z,w): \im{w}<r(\alpha, N)\}\big)\times\CC, \label{E:1stHalf_met}
\end{align}
which establishes one half of the estimate \eqref{E:approach_met}.
\smallskip

We shall now deduce the desired lower bound. As in the proof of Theorem~\ref{T:kernel},
we set $a:=\min\{f^{-1}(1),1\}$ and $\Delta:=D(0, a)\times\D$. Also, for reasons analogous
to those in the proof of Theorem~\ref{T:kernel} (or in the previous paragraph), we have:
\begin{align}
  \met{F}(z,s+it;\,\xi)\,&=\,\met{F}(z,it;\,\xi) \; \; \forall (z,s+it;\,\xi)\in \OM_F\times\CC
  																		\label{E:trans_inv_met} \\
						&\geq\,\delta \met{\loc}(z,it;\,\xi) \; \; 
				\forall (z,it;\,\xi)\in \big(\OM_F\cap\big(\tfrac{1}{2}\Delta\big)\big)\times\CC.
				\label{E:loc-to-Delta_met}
\end{align}
The second inequality follows from Result~\ref{R:key} applied to $\OM_F$,
taking $U=\Delta$ and $V=\frac{1}{2}\Delta$\,---\,the applicability of this result
being, as before, due to Theorem~\ref{T:kernel}-$(a)$.
\smallskip

Now, \textbf{fix} a point $(z,it)\in \loc$, and let $\xi\in \CC\setminus\{(0,0)\}$. In view
of \eqref{E:loc-to-Delta_met}, we need to find a lower bound for $\met{\loc}(z,it;\,\xi)$. 
This quest for a lower bound splits into two cases. In the argument below, $B>0$ will
denote a constant
whose value is independent of the variables involved, whose actual value is not of interest, and
which may change from line to line.\smallskip

\noindent{\textbf{Case~1.} \emph{$\xi\in \CC\setminus\{(0,0)\}$ such that $\xi_2\neq 0$.}}
\vspace{0.3mm}

\noindent{Consider the function
\[
  \func(\zt,w)\,:=\,-\frac{8it^3(w-it)}{\xi_2(w+it)^3} \; \; 
  \forall (\zt,w)\in \loc.
\]
It is easy to check that $\func$ belongs to the set occurring on the right-hand
side of the equation that defines 
$\jay_{\loc}(z,it;\,\xi)$. In terms
of the notation of Lemma~\ref{L:L_2_norm_alt},
\begin{equation}\label{E:func_psi}
  |\func|^2\,\leq\,\frac{128}{|\xi_2|^2}\,|\phi(\bcdot\,;1,t)|^2 + \frac{128t^2}{|\xi_2|^2}\,|\phi(\bcdot\,;0,t)|^2.
\end{equation}
Let $r_0>0$ be the constant given by Lemma~\ref{L:L_2_norm_alt}. Let us now consider
$(z,it)\in \loc$ such that $0<t<r_0$. Then, in view of \eqref{E:func_psi},
Lemma~\ref{L:L_2_norm} gives us
\[
  \|\func\|^2_{\leb{2}(\loc)}\,\leq\,\frac{B}{|\xi_2|^2}t^4(f^{-1}(t)\big)^2.
\]
for some constant $B>0$.
Therefore, by the definition of $\jay_{\loc}(z,it;\,\xi)$
in \eqref{E:B-F_aux}, clearly
\begin{multline}\label{E:jay_1st-case}
  \jay_{\loc}(z,it;\,\xi)\,\leq\,\frac{B}{|\xi_2|^2}\,t^4\big(f^{-1}(t)\big)^2 \\
  \; \; \forall (z,it)\in \loc\cap\{(z,w):\im{w}<r_0\} \text{ and $\xi: \xi_2\neq 0$}.
\end{multline}
Let us define
\begin{equation}\label{E:tau}
  \tau(\alpha, N)\,:=\,\min\{r_0, r(\alpha, N)\}
\end{equation}
where $r(\alpha, N)$ is as given by the \textsc{Fact} stated at the beginning of this proof.
Now, the latter parameter is precisely the one provided by the proof of Theorem~\ref{T:kernel}
and which is introduced just before the estimate \eqref{E:approach}. Therefore, we have,
by \eqref{E:approach}:
\[
  \frac{1}{K_F(z,it)}\,\geq\,(1/C_1)t^2\big(f^{-1}(t)\big)^2 \; \;
  \forall (z,it)\in \apprh\cap\{(z,w): \im{w}<r(\alpha, N)\}.
\]
From the latter inequality, \eqref{E:jay_1st-case}, the Bergman--Fuchs identity
\eqref{E:B-F_met}, and \eqref{E:loc-to-Delta}, we get
\begin{multline}\label{E:met_1st-case}
  \met{\loc}(z,it;\,\xi)\,\geq\,\frac{\delta|\xi_2|^2}{B\bcdot C_1}t^{-2} \\
  \; \; \forall (z,it)\in \apprh\cap\big(\tfrac{1}{2}\Delta\big)\cap\{(z,w):\im{w}<\tau(\alpha, N)\} 
  \text{ and $\xi: \xi_2\neq 0$}. 
\end{multline}}
\vspace{0.1mm}

\noindent{\textbf{Case~2.} \emph{$\xi\in \CC\setminus\{(0,0)\}$ such that $\xi_1\neq 0$.}}
\vspace{0.3mm}

\noindent{Consider the function
\[
  \funk(\zt,w)\,:=\,-\frac{4(\zt-z)t^2}{\xi_1(w+it)^2} \; \; 
  \forall (\zt,w)\in \loc.
\]
It is easy to verify that $\funk$ belongs to the set occurring on the right-hand
side of the equation that defines 
$\jay_{\loc}(z,it;\,\xi)$. In this case, in terms
of the notation of Lemma~\ref{L:L_2_norm},
\begin{equation}\label{E:funk_psi}
  |\funk|^2\,\leq\,\frac{32}{|\xi_1|^2}\,|\psi(\bcdot\,;0,2,1,t,1)|^2
  + \frac{32}{|\xi_1|^2}\,|\psi(\bcdot\,;1,2,0,t,z)|^2.
\end{equation}
As before, let us first restrict
$(z,it)$ to $\loc$ such that $0<t<r_0$, where $r_0$ is as given by Lemma~\ref{L:L_2_norm}.
Given \eqref{E:funk_psi},
this lemma implies:
\[
  \|\funk\|^2_{\leb{2}(\loc)}\,\leq\,\frac{B}{|\xi_1|^2}\,t^2\big(f^{-1}(t)\big)^4,
\]
for some constant $B>0$. Therefore, by the definition of $\jay_{\loc}(z,it;\,\xi)$
in \eqref{E:B-F_aux},
\begin{multline}\label{E:jay_2nd-case}
  \jay_{\loc}(z,it;\,\xi)\,\leq\,\frac{B}{|\xi_1|^2}\,t^2\big(f^{-1}(t)\big)^4 \\
  \; \; \forall (z,it)\in \loc\cap\{(z,w):\im{w}<r_0\} \text{ and $\xi: \xi_1\neq 0$}.
\end{multline}
Defining $\tau(\alpha, N)$ exactly as in \eqref{E:tau} and arguing exactly as in
Case~1, we get
\begin{multline}\label{E:met_2nd-case}
  \met{\loc}(z,it;\,\xi)\,\geq\,\frac{\delta|\xi_1|^2}{B\bcdot C_1}\big(f^{-1}(t)\big)^{-2} \\
  \; \; \forall (z,it)\in \apprh\cap\big(\tfrac{1}{2}\Delta\big)\cap\{(z,w):\im{w}<\tau(\alpha, N)\}
  \text{ and $\xi: \xi_1\neq 0$}.
\end{multline}}

To complete the proof, we first note that for each relevant $(z,it)$,
\eqref{E:met_1st-case} and \eqref{E:met_2nd-case} give two different
lower bounds for $\met{\loc}(z,it;\,\bcdot)$ on the set
$\{\xi\in T^{1,0}_{(z,it)}(\OM_F\cap (\frac{1}{2}\Delta)): \xi_1\neq 0 \text{ and } \xi_2\neq 0\}$.
From this, it follows that
\begin{multline}\label{E:met_low-bound_provisional}
   \met{\loc}(z,it;\,\xi)\,\geq\,\frac{\delta}{2B\bcdot C_1}\big(\big(f^{-1}(t)\big)^{-2}|\xi_1|^2
   +t^{-2}|\xi_2|^2\big) \\
  \; \; \forall (z,it;\,\xi)\in 
  \big(\apprh\cap\big(\tfrac{1}{2}\Delta\big)\cap\{(z,w):\im{w}<\tau(\alpha, N)\}\big)\times \CC.
\end{multline}
Lowering the value of the constant $r_0$ that occurs in \eqref{E:tau}, if necessary, we
may assume that
\[
  \OM_F\cap\{(z,w): \im{w}<r_0\}\,\subseteq\,\OM_F\cap\big(\tfrac{1}{2}\Delta\big).
\]
Then, from \eqref{E:trans_inv_met}, \eqref{E:loc-to-Delta_met} and \eqref{E:met_low-bound_provisional},
we get
\begin{align}
  \met{F}(z,w;\,\xi)
    &\geq\,\frac{\delta^2}{2B\bcdot C_1}\big(\big(f^{-1}(\im w)\big)^{-2}|\xi_1|^2
    + |\im w|^{-2}|\xi_2|^2\big)  \notag \\
    &\qquad\qquad \forall (z,w;\,\xi)\in \big(\apprh\cap\{(z,w): \im{w}<\tau(\alpha, N)\}\big)\times\CC.
    \label{E:low_bound_met}
\end{align} 
This establishes the other half of the estimate
\eqref{E:approach_met}. Raising the value of the constant $C_2>0$ introduced just prior
to \eqref{E:1stHalf_met}, if necessary, \eqref{E:approach_met} now follows
from \eqref{E:1stHalf_met}
and \eqref{E:low_bound_met}. \hfill $\Box$ 
\smallskip

\section*{Acknowledgements}\vspace{-0.1cm}
I thank the referees of this work for helpful suggestions on exposition, and for pointing
out a simpler expression for one of the conditions needed above. I also thank the referee
who drew my attention to a gap in a proof in an earlier version of this work.
This work is supported by a Swarnajayanti Fellowship (Grant No.~DST/SJF/MSA-02/2013-14)
and by a UGC CAS-II grant (Grant No. F.510/25/CAS-II/2018(SAP-I)).

\end{document}